\documentclass[11pt,leqno]{article} 
\usepackage[margin=1.1in]{geometry}
\usepackage{amsmath, graphicx, amsfonts, amsthm, amsxtra, amssymb, verbatim, makeidx}
\usepackage{subeqnarray, relsize}
\usepackage[mathscr]{euscript}
\usepackage{hyperref}
\usepackage[utf8]{inputenc}
\usepackage[american]{babel}

\usepackage{mathtools}
\usepackage{booktabs}
\usepackage{algorithm}
\usepackage{algpseudocode}
\makeindex
\makeglossary

\usepackage[style=alphabetic,citestyle=alphabetic,backend=bibtex]{biblatex}

\DeclareMathOperator{\pic}{Pic}
\DeclareMathOperator{\jac}{Jac}

\DeclareMathOperator{\GL}{GL}
\DeclareMathOperator{\gl}{GL}

\DeclareMathOperator{\rk}{rk}

\DeclareMathOperator{\res}{res}

\newcommand{\dd}{\mathrm{d}}
\DeclareMathOperator{\codim}{codim}
\DeclareMathOperator{\gr}{Gr}

\newcommand{\p}{\Pp^1}
\newcommand{\pp}{\Pp^2}
\newcommand{\ppp}{\Pp^3}

\newcommand{\inv}{^{-1}}

\newcommand{\tos}{\twoheadrightarrow}
\newcommand{\toi}{\hookrightarrow}
\newcommand{\isoto}{\overset{\sim}{\to}}

\newcommand{\Cc}{\mathbb{C}}

\newcommand{\Ff}{\mathbb{F}}

\newcommand{\Nn}{\mathbb{N}}
\newcommand{\Qq}{\mathbb{Q}}
\newcommand{\Pp}{\mathbb{P}}

\newcommand{\Zz}{\mathbb{Z}}

\renewcommand{\H}{\mathrm{H}}
\newcommand{\PH}{\mathrm{PH}}

\newcommand{\cL}{\mathcal{L}}

\newcommand{\co}{\mathcal{O}}
\newcommand{\cq}{\mathcal{Q}}

\def\sfK{{\sf K}}                    %an extension of \k
\def\an{{\rm an}}
\def\prim{{\rm  prim}}                  %primitive part
\def\tI{\tilde{I}}

\def\Z{\mathbb{Z}}                   %Integer  numbers
\def\Q{\mathbb{Q}}                   %Rational  numbers
\def\C{\mathbb{C}}                   %Complex numbers
\def\N{\mathbb{N}}                   %natural numbers
\def\dR{{\rm dR}}                    %The subindex dR standing for de Rham cohomology.
\def\F{{\cal F}}                     %A foliation
\def\Tr{{\rm Tr}}                      %Trace map in Algebraic de Rham cohomology
\def\GL{{\rm GL}}                %The liner group
\def\k{{\sf k}}                     %Arbitrary field
\def\Gal{{\rm Gal}}              %The Galois group
\def\Ff{\mathbb{F}}                  %Finite field

\def\O{{\cal O}}                     %ring of integers of a number field
\def\gcd{{\rm gcd}}                  %greatest common divisor
\def\bark{\overline{\k}} % algebraic closure of \k
\renewcommand{\F}{\mathrm{F}} % filtration
\DeclareMathOperator{\Hdg}{Hdg}

\newcommand{\del}{\partial}

\usepackage{color}
\newcommand{\emre}[1]{\strut{\color{red} $\clubsuit$}%
  \marginpar{\color{red}\footnotesize\raggedright  \textbf{Emre:} #1}}
\newcommand{\hossein}[1]{\strut{\color{blue} $\spadesuit$}%
  \marginpar{\color{blue}\footnotesize\raggedright  \textbf{Hossein} #1}}

\renewcommand{\emre}[1]{}
\renewcommand{\hossein}[1]{}

\newtheorem{theorem}{Theorem}
\newtheorem{example}[theorem]{Example}
\newtheorem{corollary}[theorem]{Corollary}
\newtheorem{question}[theorem]{Question}
\newtheorem{definition}[theorem]{Definition}

\newtheorem{lemma}[theorem]{Lemma}

\newtheorem{proposition}[theorem]{Proposition}
\newtheorem{remark}[theorem]{Remark}

\numberwithin{equation}{section}
\numberwithin{theorem}{section}
\numberwithin{figure}{section}
\numberwithin{algorithm}{section}

\begin{document}

\begin{center}
{\LARGE\bf On reconstructing subvarieties from their periods
}
\\
\vspace{.25in} {\large {\sc Hossein Movasati}} 
\footnote{
Instituto de Matem\'atica Pura e Aplicada, IMPA, Estrada Dona Castorina, 110, 22460-320, Rio de Janeiro, RJ, Brazil,
\url{www.impa.br/~hossein}, {\tt hossein@impa.br}}
 {\large and {\sc Emre Can Sert\"oz}}
\footnote{
  Max Planck Institute for Mathematics in the Sciences, MPI MiS, Inselstra{\ss}e 22, Leipzig 04103, Germany
  \url{www.emresertoz.com}, {\tt emre@sertoz.com}}
\end{center}
\begin{abstract}
We give a new practical method for computing subvarieties of projective hypersurfaces. By computing the periods of a given hypersurface $X$, we find algebraic cohomology cycles on $X$. On well picked algebraic cycles, we can then recover the equations of subvarieties of $X$ that realize these cycles. In practice, a bulk of the computations involve transcendental numbers and have to be carried out with floating point numbers. However, if $X$ is defined over algebraic numbers then the coefficients of the equations of subvarieties can be reconstructed as algebraic numbers. A symbolic computation then verifies the results.

As an illustration of the method, we compute generators of the Picard groups of some quartic surfaces. A highlight of the method is that the Picard group computations are proved to be correct despite the fact that the Picard numbers of our examples are not extremal.
\end{abstract}
\section{Introduction}

The Hodge conjecture asserts that on a smooth projective variety over $\Cc$, the $\Qq$-span of cohomology classes of algebraic cycles and of Hodge cycles coincide~\cite{Deligne-HodgeConjecture}. One difficulty of this conjecture lies in the general lack of techniques that can reconstruct algebraic cycles from their cohomology classes. In this article, we take a computational approach to this reconstruction problem and develop Algorithm~\ref{alg:ideal}. The highlight of this algorithm is its practicality. We have a computer implementation of the method that allows us to give rigorous Picard group computations, see Section~\ref{sec:conic}. 

Our approach parallels that of~\cite{griffiths-harris--IVHS2} where it is proven that the periods of a general curve in a high degree surface in $\ppp$ are sufficient to reconstruct the equations of the curve. Because period computations are expensive in practice, our intended applications are towards low degree hypersurfaces where the periods give partial, but substantial, information. The computation of the Picard rank of quartic surfaces in~Section~\ref{sec:implementation} is an example. On another note, we expect that one can experiment with reconstructing Hodge cycles in hypersurfaces where the Hodge conjecture is not known, using~\cite{sertoz18,lairez-sertoz} and the arguments here.

\subsection{Outline of the method}

For a smooth hypersurface $X \subset \Pp^{n+1}_{\Cc}$ of degree $d$, Griffiths residues~\cite{griffiths--periods} establishes a connection between homogeneous polynomials $p$ on $\Pp^{n+1}_\C$ and cohomology classes $\omega_p$ on $X$, see Section~\ref{sec:notation} for a precise statement.  Let $S_u$ be the space of degree $u \in \Zz$ homogeneous forms on $\Pp^{n+1}_\C$, with $S_u=0$ if $u<0$.  For a topological cycle $\delta \in \H_n(X,\Zz)$ and $u,v \in \Zz$ with $u+v = (n/2+1)d-n-2$ we consider the following pairing:
\begin{align}
  S_u \times S_v &\to \Cc, \\
  (p,q) &\mapsto  \int_\delta \omega_{pq}, \nonumber
\end{align}
and the associated map $\varphi_{\delta,u} \colon S_u \to S_v^\vee$ with kernel $\tilde I_{\delta,u}$. See Section~\ref{sec:implementation} and Algorithm~\ref{alg:ideal} in particular on how to compute $\tilde I_{\delta,u}$.

If $\delta=[Y]$ for a subvariety $Y \subset X$ then the ideal $\tilde I = \bigoplus_{u\in \Nn} \tilde I_{\delta,u}$ contains the ideal $I(Y)$ of polynomials vanishing on $Y$ (Proposition~\ref{prop:containment}). However, the inclusion can be strict. When $Y$ is a complete intersection in $\Pp^{n+1}_\C$ an observation of~\cite{ananyo17} implies that the ideals $\tilde I_{\delta}$ and $I(Y)$ coincide up to some explicit degree~$m$~(Corollary~\ref{cor:eqns}). This property is not limited to complete intersections, as we demonstrate by recovering the ideals of twisted cubics on quartic surfaces, see Sections~\ref{sec:perfect} and~\ref{sec:twisted_cubics_revisited}. We currently do not know how to compute $m$ \emph{a priori} --- unless $Y$ is a complete intersection --- but $m$ is often easy to find once $\tilde I$ is computed. 

\subsection{Applications to quartic surfaces}

Starting with the defining equation of the hypersurface $X$, we can numerically approximate its periods~\cite{sertoz18}, deduce Hodge cycles $\delta$ on the cohomology of $X$~\cite{lairez-sertoz}, and compute the associated ideals $\tilde I_\delta$~(Section~\ref{sec:implementation}). In this manner, it is now possible to construct \emph{rigorous} lower bounds for the group of Hodge cycles on $X$. In order to demonstrate the method, we prove that the following two quartic surfaces
\begin{align*}
  X &= Z(x^4 + x^3z - xy^3 + y^4 + z^4 + w^4) \subset \ppp_{\Cc},\\
  X &= Z(5x^4 - 4x^2zw + 8y^4 - 5z^4 + 4zw^3) \subset \ppp_{\Cc}
\end{align*}
have Picard numbers~$8$ and $14$ respectively~(Section~\ref{sec:conic}). 

The novelty of the method is that it bypasses the need to symbolically search for points in Hilbert schemes of $X$, which is prohibitively expensive to carry out with present day computers. Once the periods of $X$ are computed, simple linear algebra gives (an approximation of) the equations cutting out subvarieties on $X$. 

Explicit determination of algebraic cycles on the surface $X$ gives a lower bound on its Picard group. As our method places these classes inside the integral homology of $X$, we can saturate the lattice they generate. To claim that this saturation is the Picard group we need to find a (sharp) upper bound on the Picard number. We search for this complementary upper bound by reducing to finite characteristic, see~\cite{abbott-10} or~\cite{costa-sertoz}. 

% in lairez sertoz, almost everything had Picard groups generated by low degree smooth rational curves, therefore these lower bounds can be ..  % finding conics and twisted cubics etc are hard!! this could be seen as a method in effective algebraic geometry
%The problem of constructing equations of subvarieties in hypersurfaces can be expressed as solving a system of polynomial equations. For finding lines in surfaces, this system is small enough to solve using symbolic methods. But already determining conics in surfaces becomes a serious challenge in this way, let alone finding twisted curves. Our approach replaces the problem of solving these large systems of polynomials with performing linear algebra with periods.

\subsection{Perfect Hodge cycles}

This work prompts the question: Which algebraic subvarieties are reconstructible from their periods in general? A major step forward would be to determine if all the equations obtained from the periods of an algebraic cycle are caused by a representative of that cycle. 

The notion of a \emph{perfect} Hodge class given in Definition~\ref{def:perfect} makes this requirement precise. We demonstrate by way of example that knowing if a cycle is perfect has strong implications for the reconstruction step: Proposition~\ref{30july2019} states that if the class of a twisted cubic in a quartic surface is perfect then this twisted cubic can be reconstructed from its periods.

We ask in Question~\ref{q:perfect} if all algebraic classes are perfect, believing that an investigation of this problem will be fruitful.

\subsection{Fields of definition of algebraic cycles}

In the last section of this article we study smooth projective varieties in general, only with the assumption that they are defined over a subfield $\k \subset \Cc$ of the complex numbers. In this generality, we can not use Griffiths' basis for cohomology and our method does not generalize. Nevertheless, we want to facilitate the search for algebraic cycles with the aid of periods. 

We ask ``Given the periods of an algebraic subvariety $Y$ on $X$, what can we recover about $Y$?'' If $\k_Y/\k$ is the field generated by the periods of $Y$ over $\k$ then there exists an integer $m$ and an algebraic cycle $Z$ on $X$ defined over $\k_Y$ such that the cohomology class of $Z$ is equal to the cohomology class of $mY$~(Proposition~\ref{prop:vistachinesa2018}). In practice, having bounds on $m$ for which such a cycle $Z$ exists is desirable. Working with divisors $Y \subset X$ and assuming $\H^1(X,\co_X)=0$, we compute an explicit bound for $m$~(Theorem~\ref{thm:linear_system}). It follows that if, in addition, $H^2(X,\Zz)$ has no torsion and $X(\k)\ne \emptyset$ then we may take $m=1$. The point of this result is that it simplifies the search for explicit divisors whose classes generate the Picard group of a given variety.

The foundational results in the last section are well-known, but we hope our exposition, with an eye towards computation, will be of help in developing the techniques further.

%In case the periods of $Y$ are not readily accessible, we can give bounds on the degree of the field extension $\k_Y/\k$ that depend \emph{only on $X$}~(Lemma~\ref{lem:bound2}, Theorem~\ref{thm:upper_bound}). We have the following application: If $X \subset \ppp_\k$ is a smooth surface of degree $d$ with Picard number~$2$ then there exists a curve $C \subset X_\Cc$, defined over $\k'/\k $ with $[\k':\k] \le 2d$, such that $[C]$ and the polarization generate $\pic(X)$ (Proposition~\ref{prop:rk2_surface}).

\subsection{Notation}

Throughout this paper, $\k$ will be a subfield of the complex numbers $\C$. We will denote by $S=\k[x_0,\dots,x_{n+1}]$ the homogeneous coordinate ring of $\Pp^{n+1}_{\k}$ and by $S_u \subset S$ the subspace of degree $u \in \Zz$ homogeneous forms in $S$. For a smooth projective variety $X/\k$ the symbols $\H^m_{\dR}(X/\k)$, $\H^{p,q}_{\dR}(X/\k)$ refer to the algebraic de~Rahm cohomology of $X$~\cite{gro66}. The Hodge filtration on cohomology is denoted by $F^\ell \H^m_{\dR}(X/\k) \simeq \bigoplus_{p \ge \ell}\H^{p,m-p}$.  We will write $\H^m(X,\Zz)$, $\H^m(X,\Qq)$, $\H_m(X,\Zz)$ and $\H_m(X,\Qq)$ for the singular homology and cohomology groups of the underlying complex analytic variety $X^{\an}$ of $X$. The canonical pairing between a cohomology class and a homology class will be denoted by the integration symbol.

\subsection*{Acknowledgements}

We are grateful to Pierre Deligne for his detailed comments on an earlier draft of this paper which shaped the present version. We also benefited from our conversations with Daniele Agostini, Moritz Firsching, Uwe Nagel, Frank Schreyer and Bernd Sturmfels. The second author would like to thank IMPA for additional funding and hospitality during his visit.
We are indebted to the anonymous referee for giving valuable insight and contributing Proposition~\ref{sec:T_in_high_d}.

\section{Reconstructing equations of effective Hodge cycles}
Let $X \subset \Pp^{n+1}_\k$, $\k \subset \C$, be an even dimensional smooth hypersurface given by the zero set of a homogeneous polynomial $f\in S=\k[x_0,\dots,x_{n+1}]$, and $Y \subset X$ a subvariety of dimension~$n/2$. By \emph{periods of $Y$} we mean the following linear map:
\begin{equation}\label{eq:periods}
  \H_{\dR}^{n/2,n/2}(X/\k) \to \Cc : \omega \mapsto \int_Y \omega.
\end{equation}
In this section we construct an ideal $\tilde I_{[Y]}$ from the periods of $Y$ and study how much of the ideal of $Y$ can be recovered from $\tilde I_{[Y]}$. Combining this with previous work to compute periods, we find the equations for conics and twisted cubics inside quartic surfaces. We also ask if the ideal $\tilde I_{[Y]}$ is entirely geometric in origin, that is, if it is spanned by the ideals of subvarieties of $X$ whose primitive cohomology classes are proportional to that of $Y$. We demonstrate how such a result can be used (or falsified) with the example of a twisted cubic on a quartic surface.

A quotient $I_{\delta}$ of the ideal $\pi(\tI_\delta)$ appears in the literature in the context of the tangent spaces of the Hodge and Noether--Lefschetz loci~\cite[\S 6.2]{voisin-2007-volII}, \cite{ananyo17}, \cite{loyola18}, \cite[\S 11]{movasati-loyola--hodge2}. See Section~\ref{sec:first_consequences} for the relation between these two ideals.

\subsection{An ideal attached to Hodge cycles}\label{sec:notation}

Let $X=Z(f) \subset \Pp^{n+1}_\k$ be a hypersurface of degree $d$ with $n$ even. Let $h^{n/2} \in \H^{n}(X,\Zz)$ be the polarization on $X$, that is, the Poincar\'e dual of a $(n/2+1)$-plane section of $X$. For any ring $\sfK$ we will write $\PH^{n}(X,\sfK) \subset \H^{n}(X,\sfK)$ for the primitive cohomology groups, i.e., the space orthogonal to $h^{n/2}$. 

Denote by $S = \k[x_0,\dots,x_{n+1}]$ the coordinate ring of $\Pp^{n+1}_\k$ and for each $u \in \Zz$ let $S_u \subset S$ be the space of homogenous polynomials of degree $u$, with $S_u=0$ if $u <0$. For a variety $Y \subset \Pp^{n+1}_\k$ let $I(Y) \subset S$ be the ideal of $Y$. 
%Let $\vol_{\Pp^{n+1}} = \sum_{i=0}^{n+1} (-1)^ix_i \dd x_0 \dots \widehat{\dd x_i} \dots \dd x_{n+1} \in \H^0(\Pp^{n+1},\Omega_{\Pp^{n+1}/\k}(n+2))$ denote the standard volume form on $\Pp^{n+1}$, where hat denotes omission. 
Let $$N:=d-n-2$$ and for each $\ell \ge 0$ denote the Griffiths residue maps~\cite{griffiths--periods} as follows:
\begin{align}
  S_{N+\ell d} &\tos F^{n-\ell}\H^n_\dR(X)=\bigoplus_{i \le \ell} \PH^{n-i,i}(X), \\
  p &\mapsto \omega_p := \res \left( \frac{p}{f^{\ell+1}}
  %\vol_{\Pp^{n+1}}
 \Omega
  \right), \nonumber
\end{align}
where 
$$
\Omega:= \sum_{i=0}^{n+1} (-1)^ix_i \dd x_0 \dots \widehat{\dd x_i} \dots \dd x_{n+1}
$$
and hat denotes omission. In this paper we will identify Hodge cycles in homology and cohomology using Poincar\'e duality and write
$$
\Hdg^{n/2}(X)=\left\{
\delta\in \H_n(X,\Z) \Big| \int_{\delta}F^{n/2+1}\H^n_\dR(X)=0  
\right\}.
$$
%\subsection{Idea}
\begin{definition}\label{def:phi}
For each $u \in \Zz$ let $v=N+\frac{n}{2}d-u$ and, for each Hodge class $\delta \in \Hdg^{n/2}(X)$, consider the pairing:
\begin{equation}
  S_u \times S_v \to \Cc : (p,q) \mapsto \int_\delta \omega_{pq}.
\end{equation}
Expressed differently, we have a map $\varphi_{\delta,u} \colon S_u\otimes \Cc \to S_v^\vee \otimes \Cc$ and the kernel 
$$
\tI_{\delta,u} := \ker \varphi_{\delta,u}.
$$
Let $\tI_\delta = \bigoplus_{u \ge 0} \tI_{\delta,u} \subset S$ and observe that $\tI_\delta$ is an ideal. 
\end{definition}

\begin{proposition}\label{prop:containment}
  If $\delta = [Y]$ for a subvariety $Y \subset X$ then $I(Y) \subset \tilde I_{[Y]}$.
\end{proposition}
\begin{proof}
  Since $S_{<0}=0$, for any $u>N+\frac{n}{2}d$ we trivially have $\tilde I_{\delta,u}=S_u$. For $u\le N+\frac{n}{2}d$,  $p \in I(Y)_u$ and any $q \in S_v$ we have $pq \in I(Y)_{N+\frac{n}{2}d}$. Therefore, it will be sufficient to show that for any $p \in I(Y)_{N+\frac{n}{2}d}$ we have $\int_Y \omega_p = 0$.

  Since $\omega_p \in \F^{n/2}\H^n_\dR(X) = \bigoplus_{i \le n/2} \H^{n-i,i}(X)$ we can write $\omega_p = \sum_{i \ge n/2} (\omega_p)^{i,n-i}$ respecting the Hodge structure. Since $\delta=[Y]$ is algebraic, $\int_{\delta} \omega_p = \int_Y (\omega_p)^{n/2,n/2}$. As computed in \cite{carlson-griffiths}, the class $(\omega_p)^{n/2,n/2}$ admits a simple representation as a \v{C}eck cohomology class in $\PH^{n/2}(X,\Omega^{n/2}_{X/k})$. Let us write $U_j \subset X$ for the open locus where the $j$-th derivative $f_j := \del f/\del x_j$ does not vanish, and let $\iota_j$ be the operator which contracts forms by $\del/\del x_j$. Given a tuple $J=(j_0,\dots,j_q) \in \{0,\dots,n+1\}^{q+1}$ write $U_J = \cap_{j \in J}U_j$, $\iota_J = \iota_{j_q} \circ \dots \circ \iota_{j_0}$, $\Omega_J = \iota_J \Omega$ and $f_J = \prod_{j \in J} f_j$. We then have the following equivalence of \v{C}eck cocycles (modulo coboundary):
\begin{equation}\label{eq:top_form}
  (\omega_p)^{n/2,n/2} \equiv \left(p \frac{\Omega_J}{f_J}, U_J\right)_{|J| = n/2 +1}.
\end{equation}
When $p\in I(Y)$, the form representing $(\omega_p)^{n/2,n/2}$ will pullback to zero on $Y$. Therefore $\int_Y (\omega_p)^{n/2,n/2} = 0$.
\end{proof}

\subsection{First consequences}\label{sec:first_consequences}
Let $S = \k[x_0,\dots,x_{n+1}]$ be the coordinate ring of $\Pp^{n+1}_\k$ and  $\jac(f) = (f_0,\dots,f_{n+1}),\ \ f_i:=\frac{\partial f}{\partial x_i}$ be the Jacobian ideal of $f$, $R=S/\jac(f)$ and $\pi\colon S \to R$ the quotient map. We will write $R_u = \pi(S_u)$ and for each $\delta \in \Hdg^{n/2}(X)$ write $I_{\delta}=\pi(\tI_\delta)$.
\begin{lemma}\label{lem:ItoK}
  For any $\delta \in \Hdg^{n/2}(X)$ we have $\tI_\delta = \pi\inv(I_\delta)$.
\end{lemma}
\begin{proof}
  We need to show that $\jac(f) \subset \tI_\delta$. It will be sufficient to show that if $p \in \jac(f) \cap S_{N+\frac{n}{2}d}$ then $\omega_p$ annihilates $\delta$.  The residue map $\res \colon S_{N+\ell d} \to F^{n-\ell}\H^n_\dR(X)$ can be factored into an isomorphism taking the form $R_{n+\ell d} \isoto F^{n-\ell}\H^n_\dR(X)/F^{n-\ell+1}\H^n_\dR(X) : p \mapsto (\omega_p)^{n-\ell,\ell}$. Combine this observation with the fact that $F^{n/2+1}\H^n_\dR(X)$ annihilates $\Hdg^{n/2}(X)$.
\end{proof}

\noindent For $\delta \in \Hdg^{n/2}(X)$ write $\delta_\prim \in \PH^{n/2}(X,\Qq)$ for the primitive part of $\delta$. 

\begin{lemma}\label{lem:parallel_prim}
  For any $\delta, \delta' \in \Hdg^{n/2}(X)$ we have $I_\delta= I_{\delta'}$ if $\Qq\langle \delta_{\prim} \rangle = \Qq\langle \delta'_{\prim} \rangle$.
\end{lemma}
\begin{proof}
  Since every form $\omega_p$ annihilates the polarization $h^{n/2}$, we have $I_\delta=I_{\delta_\prim}$. On the other hand, the kernel of the linear form $p \mapsto \int_{\delta_{\prim}} \omega_p$ depends only on the line spanned by the class $\delta_{\prim}$.
\end{proof}

Suppose $Y \subset X$ is a complete intersection of dimension $n/2$ in $\Pp^{n+1}_\k$ and write $Y=Z(g_0,\dots,g_{n/2})$. Assuming that $d_i = \deg g_i < d$ we can find homogeneous forms $h_0,\dots,h_{n/2} \in S$ such that $f = g_0h_0 + \dots + g_{n/2}h_{n/2}$.

\begin{proposition}[Proposition~2.14~\cite{ananyo17}]\label{prop:ananyo}
  With $Y \subset X$ as above, we have
  \[
    I_{[Y]}=(\pi(g_0),\dots,\pi(g_{n/2}),\pi(h_0),\dots,\pi(h_{n/2})).
  \]
\end{proposition}
\begin{corollary}
\label{29july2019}
We have
 \[
\tI_{[Y]} = (g_0,\dots,g_{n/2},h_0,\dots,h_{n/2},f_0,\dots,f_{n+1}).
\]
\end{corollary}
\begin{proof}
This follows from Proposition~\ref{prop:ananyo} and Lemma~\ref{lem:ItoK}.
\end{proof}
\begin{corollary}\label{cor:eqns}
  If there exists an $m \in \N$ such that for all $i$, $d_i < d - m$ then $\tI_{[Y],\le m} = I(Y)_{\le m}$. 
\end{corollary}
\begin{proof}
Since $\deg h_i = d-d_i$ and $\deg f_i = d-1$ the result follows from Corollary \ref{29july2019}. 
\end{proof}

\begin{example}\rm
  Suppose $Y \subset X$ is an $\frac{n}{2}$-plane and $\deg X \ge 3$. Then, we can recover the equations cutting out $Y$ from its Hodge class, since Corollary~\ref{cor:eqns} implies $I(Y)_1 = \tI_{[Y],1}$.
\end{example}

\begin{example}\rm
  Suppose $n=2$ and $C \subset X \subset \ppp_\k$ is a smooth conic ($d_1=1,\ d_2=2$). If $d \ge 5$ then by Corollary~\ref{cor:eqns} the ideal cutting out $C$ is generated by $\tI_{[C],\le 2}$. However, if $d=4$ then $\dim \tI_{[C],1}=1$ but $\dim \tI_{[C],2}=6$ although $\dim I(C)_2 = 5$. The interference with the quadrics is predicted by Lemma~\ref{lem:parallel_prim}: if $H$ is the plane containing $C$ then $H \cap X = C \cup C'$, where $C'$ is a possibly singular conic with $[C]_{\prim}=-[C']_{\prim}$. Indeed, Proposition~\ref{prop:ananyo} implies $\tI_{[C],2} = I(C)_2 + I(C')_2$.
\end{example}

\subsection{Perfect Hodge classes and the twisted cubic}\label{sec:perfect}

Lemma~\ref{lem:parallel_prim} gives a natural source for polynomials appearing in $\tilde I_{\delta}$. We distinguish the Hodge classes for which this is the only source.

\begin{definition}\label{def:perfect}
  A Hodge class $\delta\in \Hdg^{n/2}(X)$ is called \emph{perfect at level $m$} if there exists $Y_1,\dots,Y_s \subset X$ of codimension $n/2$ such that $\Qq\langle [Y_i]_{\prim} \rangle = \Qq\langle \delta_{\prim} \rangle$ and $I_{\delta,\le m} = \sum_{i=1}^s \pi(I(Y_i)_{\le m})$. If $I_{\delta} = \sum_{i=1}^s \pi(I(Y_i))$ then $\delta$ is \emph{perfect}.
\end{definition}

\begin{example}\rm
  Hodge classes supported on a complete intersection variety $Y=Z(g_0,\dots,g_{n/2})$ with $\deg g_i < d$ are perfect. In fact, the inclusion of the elements $g_i,h_i \in I_\delta$ in Proposition~\ref{prop:ananyo} uses the following observation: Let $Y_i=Z(g_0,\dots,g_{i-1},h_i,g_{i+1},\dots,g_{n/2}) \subset X$ and note that $Y \cup Y_i = X \cap Z(g_0,\dots,g_{i-1},g_{i+1},\dots,g_{n/2})$. This implies that $[Y]_{\prim} = -[Y_i]_{\prim}$. Now apply Lemma~\ref{lem:parallel_prim}.
\end{example}

\begin{question}\label{q:perfect}
  Are all algebraic Hodge classes perfect?
\end{question}

If the answer to this question is negative, it would be beneficial to have a general principle to detect if a given effective algebraic class is perfect. As we demonstrate below, it would then be possible to determine if an algebraic subvariety representing a given perfect class can be reconstructed from its periods.

\subsubsection{Twisted cubics in quartic surfaces}

We now consider the simplest non-trivial non-complete intersection: a twisted cubic $T$ in a smooth quartic surface $X \subset \ppp_\k$. We prove below that $T$ is perfect at level $2$ if and only if $I(T)$ can be recovered from $\tilde I_{[T]}$ (i.e. $\dim \tilde I_{[T],2} = 3$). %For all the examples considered in Section~\ref{sec:implementation} we could recover the quadrics vanishing on such twisted cubics. This suggests that the two equivalent hypotheses in the following proposition should hold in general.

\begin{proposition}
\label{30july2019}
  Let $T$ be a twisted cubic in a smooth quartic surface $X \subset \ppp_\k$.  The class $[T]$ is perfect at level $2$ if and only if $\tilde I_{[T],1}=0$ and $\dim \tilde I_{[T],2}=3$.
\end{proposition}
\begin{proof}
  Since $I(T)_1 = 0$ and $\dim I(T)_2 = 3$ the ``if'' implication is immediate. We now assume that $[T]$ is perfect and compute $\tilde I_{[T],m}$ for $m=1,2$.

  We begin with $m=1$. If $\tilde I_{[T],1}$ is non-zero then (by hypothesis) there exists a curve $C \subset X$ lying in a plane $Z(h)$, $h \in \tilde I_{[T],1}$, such that $\Qq\langle [C]_{\prim}  \rangle = \Qq\langle [T]_{\prim} \rangle$. Since $[C]_{\prim} \neq 0$, $[C]$ is either a line, a (possibly singular) conic or an elliptic curve. The self intersection numbers of $[C]$ in each case are $-2$, $-2$ and $0$ respectively. Since $[C]_{\prim} = [C] - \frac{d}{4}h$, where $d$ is the degree of $C$, we have $[C]_{\prim}^2 = -9/4,\, -3,\, -9/4$ respectively. Whereas, $[T]_{\prim}^2 = -2-3^2/4=-17/4 $. Since $a[T]_{\prim} = b[C]_{\prim}$ for non-zero integers $a,b$ the quotient $[C]_{\prim}^2/[T]_{\prim}^2 \in \{12/17, 9/17\}$ must be the square of a rational number. As this is not the case, we conclude $\tilde I_{[T],1}=0$.

  Now we consider $m=2$. If $\dim \tilde I_{[T],2} > 3$ then there exists a curve $C \subset X$, possibly with non-reduced components, lying in a quadric surface $Q=Z(q)$ such that $q$ does not vanish on $T$, $q \in \tilde I_{[T],2}$ and  $\Qq\langle [C]_{\prim}  \rangle = \Qq\langle [T]_{\prim} \rangle$. Note that we are not free to pick any $q$ in $\tilde I_{[T],2} \setminus I(T)_2$ and so we may not assume $Q$ is smooth.

  \emph{Case I: $Q$ is smooth.} Now $C \in Q \simeq \p \times \p$ is of bidegree $(a, b)$ for some $1 \le a \le b \le 4$. Note $(a,b) \neq (1,1)$ or $(4,4)$, because the former would imply that $C$ is contained in a plane and the latter would imply $[C]_{\prim}=0$. Since both $X$ and $Q$ are smooth the curve $C$ is a local complete intersection in both, therefore the dualizing sheaf $\omega_{C/k}$ exists and can be computed by adjunction on either $X$ or $Q$~\cite[\S 6.4.2]{liu02}. In particular, equating the two expressions for the degree of $\deg \omega_{C/k}$ by computing it on $X$ and on $Q$ respectively we obtain  $[C]^2=b(a-2)+a(b-2)$. Since the degree of $C$ is $a+b$ we find $[C]_{\prim}^2 =  b(a-2)+a(b-2) - (a+b)^2/4$. As before, the quotient $[C]_{\prim}^2/[T]_{\prim}^2$ must be the square of a rational number. In the given range of possibilities of $a,b$ this can only be realized if $(a, b) = (1,2)$ or $(2,3)$ and in both cases $[C]_{\prim}^2 = -17/4 = [T]_{\prim}^2$. As these primitive classes are proportional, we conclude $[C]_{\prim} = \pm [T]_\prim$ and $[C] = \pm [T]_{\prim} + \frac{a+b}{4} h$. Since $C$ is contained in a quadric not containing $T$, the intersection number $C \cdot T$ must be an integer between 0 and 6. For $(a,b) = (1,2)$ or $(2,3)$ this is not the case. Therefore, no such $C$ exists.  

  \emph{Case II: $Q$ is an irreducible cone.} Pass to the desingularization $\tilde Q \to Q$ and let $\tilde C \subset \tilde Q$ be the proper transform of $C$. This time we have $\deg \omega_{\tilde C} \le \deg \omega_{C} = [C]^2$ and we can compute the left hand side by adjunction on $\tilde Q$. The surface $\tilde Q$ is a Hirzebruch surface of degree 2, we refer to \cite[\S V.2]{hartshorne} for general properties of such surfaces. We have $\pic (\tilde Q) = \Zz\langle e,f \rangle$ where $e^2=-2, e\cdot f=1, f^2=0$. Here $e$ represents the class of the exceptional divisor $E$ and $f$ the class of the proper transform of a line in $Q$. Let us write $\tilde c :=[\tilde C] = ae + bf$, where $b$ equals the degree of $C$ in $\ppp_\k$. 
  
  Since $\tilde C$ has no component supported on $E$ the relations $e\cdot \tilde c , f \cdot \tilde c \ge 0$ imply $b \ge 2a \ge 0$. The class of a plane section is $h=e+2f$ and since $C$ is a component of $X \cap Q$, the class $4h-\tilde c$ must also be effective. These give the additional constraints $a \le 4$ and $b\le 8$. We may assume $C$ is not planar, and since any curve of degree $\le 2$ in $Q$ is planar we conclude $b >2$. Furthermore, $C\neq X \cap Q$ thus $b < 8$. Moreover, if $b=7$ then $C$ differs from being the complete intersection $X\cap Q$ by the class of a line $L$. Thus $[L]_{\prim}$ is proportional to $[T]_{\prim}$ in $X$, which is a contradiction. This ensures $b < 7$. 
  
  Now $\deg \omega_{\tilde C} = -2(a^2 - ab + b) \le [C]^2$ and $\deg C = b$ so that $-2(a^2 - ab + b)-b^2/4  \le [C]_{\prim}^2 < 0$. We see that there are no pairs $(a,b)$ in the specified range which allows for $[C]_{\prim}^2/[T]_{\prim}^2$ to be the square of a rational number and for $[C] \cdot [T] \in \{0,\dots,6\}$.

  \emph{Case III: $Q$ is the union of two distinct planes.} This time $C=C_1+C_2$ with $C_i$'s lying in different planes and with $C$ not planar. Since each $C_i$ is a plane curve of degree $1$, $2$ or $3$ we can readily compute the possible range of values of $([C]^2,\deg(C))$. None of these pairs give a primitive class proportional to $[T]_{\prim}$.

  \emph{Case IV: $Q$ is a double plane.} Now $[C]$ must be the sum of components of $H \cap X$ for a plane $H$, but with multiplicities between $0,1,2$. As we can always pass to $2[H \cap X]-[C]$ without changing the span of the primitive part of $[C]$, we can reduce to the planar case unless $H \cap X$ has at least three components and all three coefficients $0,1,2$ must appear in their sum $[C]$. There are only three cases to consider and they are all eliminated as above.
\end{proof}

\begin{remark}\rm
  Sampling tens of thousands of quartics containing a twisted cubic, we observed that the class $[T]$ is perfect at level~$2$ on all of these examples. See~Section~\ref{sec:many_examples}.
\end{remark}

\subsubsection{Twisted cubics in higher degree surfaces}

It is clear from Definition~\ref{def:phi} that we can find more relevant equations as the degree of the of the surface increases. This suggests that it should be easier to reconstruct curves of fixed degree and genus as we view them in higher degree surfaces. This is also the nature of the result~\cite[Corollary 4.a.8]{griffiths-harris--IVHS2}. We now observe this phenomenon with twisted cubics, giving a specific bound on the degree on the surface for reconstructability.

\begin{proposition}\label{sec:T_in_high_d}
 Let $T \subset X$ be a twisted curve in a smooth surface $X=Z(f) \subset \ppp$ of degree~$d\ge 5$. The ideal of $T$ can be reconstructed from its periods in $X$. 
\end{proposition}
\begin{proof}
  We claim $\tilde I_{[T],2} = I_{T,2}$. Suppose to the contrary that there is $q \in \tilde I_{[T],2} \setminus I_{T,2}$. Pick $g \in \jac(f)$ such that $Z(g)\cap Z(q) \cap T = \emptyset$. We set $V = R/I_T$ and note $V \simeq \Cc[t_0,t_1]$ as $T\simeq \p$ is projectively normal. The linear system $W = q|_T \cdot V_{3d-9} + g|_T \cdot V_0  \subset V_{3d-3}$ has dimension $3d-7$ and no base locus. The lemma of Gieseker~\cite[p.~194]{harris--space_curves} implies that either $W \cdot V_s = V_{3d-3+s}$ or $\dim(W \cdot V_s) = \dim W + 2s = 3d-7+2s$. Therefore, if $s \ge 5$ then $W \cdot V_s = V_{3d-3+s}$.

  On the other hand, $\tilde I_{[T],2d-4} \neq R_{2d-4}$, since the primitive part of $[T]$ is non-zero. Alternatively, $R/\tilde I_{[T]}$ is Artinian--Gorenstein with socle at degree $2d-4$~\cite[Remark~4]{loyola18}. Therefore, when $s=3(2d-4)-3d-3=3d-9$ we must have $W \cdot V_s \neq V_{3d-3+s}$. Hence $3d-9 < 5$, which is a contradiction.
\end{proof}

\begin{remark}\rm
  This result and its proof was generously suggested to us by the anonymous referee. Any mistake in transcription is due to us.
\end{remark}

\section{Implementation}\label{sec:implementation}

Let $X = Z(f) \subset \Pp^{n+1}_{\k}$ be an even dimensional smooth hypersurface with $\k \subset \Cc$. For the code we have written we assume $f$ to have rational coefficients, i.e. $\k = \Qq$, although in principle algebraic coefficients would work. The examples in this section are computed using {\tt PeriodSuite}\footnote{available at \url{https://github.com/emresertoz/PeriodSuite}}. We will now describe how the computation of $\tI_{\delta,u}$ is carried out. 

Throughout we use the grevlex monomial basis for the quotient $R=S/\jac(f)$. The period computations in~\cite{sertoz18} takes $f$ as input and returns an approximation $\cq\in \Cc^{s\times s}$ of the isomorphism $R\otimes \Cc \isoto \PH^{n}(X,\Zz)\otimes \Cc$ with respect to some basis $\gamma_1,\dots,\gamma_s \in \PH^n(X,\Zz)$.

Let $\pi\colon S \to R/\jac(f)$ denote the quotient map. The product $S_u \times S_v \to S_{u+v} \tos R_{u+v}$ can be computed explicitly with respect to the grevlex monomial basis on $R$ and monomial bases for $S_u$, $S_v$. We can compute the lattice of Hodge classes inside $\H^n(X,\Zz) \simeq \Zz^{s+1}$ using these periods~\cite{lairez-sertoz}, provided sufficient precision is used. For a Hodge cycle $\delta$, let $v_\delta=(v_{\delta,i})_{i=1}^s \in \Zz^s$ such that $\delta_{\prim} = \sum_{i=1}^s v_{\delta,i} \gamma_i$.  The approximation of the vector space $\tI_{\delta,u} \subset S_u$ now reduces to numerical linear algebra~(Algorithm~\ref{alg:ideal}).

In {\tt PeriodSuite} the command {\tt PolynomialsVanishingOnHodgeCycle} returns $\tI_{\delta,u}$ for given $\delta$ and $u$\emre{currently in MHS branch}. The most expensive step is the computation of the approximation $\cq$ of the period isomorphism, everything else typically takes less than a second.

\begin{algorithm}
\caption{Computing generators for the space $\tilde I_{\delta,u}$}\label{alg:ideal}
\begin{algorithmic}[1]
\Procedure{PolynomialsVanishingOnHodgeCycle}{$f$,$\cq$,$v_\delta$, $u$}
\State $v \gets N+\frac{n}{2}\deg f-u$
\State $\{m_i\} \subset S_u$, $\{m_j'\} \subset S_v$ monomial bases for $S_u$ and $S_v$
\State $\forall i,j$, $r_{ij;k} \gets $ the $k$-th coordinate vector representing $\pi(m_im_j') \in R_{u+v}$
\State $M_i \gets (r_{i,j;k})_{j,k}$ \Comment{Matrix representing the map $m'_j \mapsto \pi(m_im'_j) \in R_{u+v}$}
\State $M_i^\delta \gets M_i \cdot \cq\cdot  v_\delta^t$ \Comment{Column vector representing the map $m'_j \mapsto \int_\delta \omega_{m_im'_j}$}
\State $M \gets (M^\delta_i)_i$ \Comment{Matrix with rows $M_i^\delta$}
\State $B \gets $ a basis for the left kernel of $M$ 
\State \textbf{return} $\{\sum_{i} b_i m_i \mid (b_i)_i \in B\} $
\EndProcedure
\end{algorithmic}
\end{algorithm}

\subsection{Conics in quartic surfaces}\label{sec:conic}

We will now demonstrate how the Picard number computations relying on approximate periods can be made rigorous in some instances. The proof of the following proposition can be automated, therefore we give another example afterwards.

Let us begin by recalling the maps $\varphi_{\delta,u}$ following Definition~\ref{def:phi} for a quartic $X=Z(f)$. Here the only non-zero maps are those that come from degrees $u=0,\dots,4$. For each $u$, the complementary degree is $v=4-u$ and the maps are defined by
\begin{eqnarray*}
  \varphi_{\delta,u} \colon R_u &\to& R_v^\vee \otimes \Cc \\
  p &\mapsto& \left( q \mapsto \int_\delta \omega_{pq} \right),
\end{eqnarray*}
where $R=S/\jac(f)$ and $\delta \in \H_2(X,\Zz)$ is a Hodge class. 

In the proofs below, $\delta$ will be the class of a conic and $u=1$. In light of Corollary~\ref{29july2019} the kernel of $\varphi_{\delta,1}$ is spanned by the equation of the plane containing the conic representing $\delta$.

% x^4 + x^3*z - x*y^3 + y^4 + z^4 + w^4
\begin{proposition}\label{prop:rk8}
  The quartic surface $X=Z(f) \subset \ppp_{\Cc}$ defined by the polynomial $f=x^4 + x^3z - xy^3 + y^4 + z^4 + w^4$ has Picard number~$8$. In fact, $X$ contains $56$ conics consisting of $28$ coplanar bitangent pairs. The classes of these conics generate $\pic(X)$.
\end{proposition}

\begin{proof}
  First, we show symbolically that $X$ contains no lines. Anticipating the reconstruction step ahead we compute the periods of $X$ to $5000$ digits~\cite{sertoz18}. We then find a rank $8$ lattice $\Lambda \subset \H^2(X,\Zz)\simeq \Zz^{22}$, together with a polarization $h \in \Lambda$; this lattice is the Picard group of $X$ with high confidence~\cite{lairez-sertoz}. There are $56$ elements $\delta \in \Lambda$ such that $\delta^2=-2$ and $h \cdot \delta =2$, these must be the classes of conics. For each class $\delta$ containing a conic, the numerical computation of $\tI_{\delta,1}$ is straightforward.

We pick the class $\delta$ of one conic and denote by $h=a_0x_0+a_1x_1+a_2x_2+a_3x_3$ a generator of $\tI_{\delta,1}$. Scaling $h$ so that $a_0=1$, we observe $a_3=0$ and that $a_1$ minimally satisfies the following equation:
  \begin{multline*}
      250111a_1^{28} + 3805704a_1^{27} + 22090752a_1^{26} + 55887424a_1^{25} + 29659840a_1^{24} - 146479104a_1^{23} - 339799776a_1^{22} \\
      - 305663232a_1^{21} + 99439104a_1^{20} + 550785848a_1^{19} + 590500464a_1^{18} + 400266240a_1^{17} + 432352000a_1^{16} \\
      + 682095104a_1^{15} + 415718400a_1^{14} - 903697024a_1^{13} - 2446500160a_1^{12} - 2753616384a_1^{11} - 1384379792a_1^{10} \\
      + 638347008a_1^9 + 1876787712a_1^8 + 1838682624a_1^7 + 1112955904a_1^6 + 435142656a_1^5 + 96773184a_1^4 +\\
      8895488a_1^3 + 411648a_1^2 + 8256a_1 + 64=0.%\\ 
%      250111a_2^{28} - 2294400a_2^{27} + 10236624a_2^{26} - 35419424a_2^{25} + 148872384a_2^{24} - 510943104a_2^{23} + 1179956080a_2^{22} \\
%      - 1843866624a_2^{21} + 1836264384a_2^{20} - 685181760a_2^{19} - 1153759232a_2^{18} + 2454706176a_2^{17} - 2311236352a_2^{16} \\
%      + 968556544a_2^{15} + 360069120a_2^{14} - 729995264a_2^{13} + 417230848a_2^{12} - 268886016a_2^{11} + 506732544a_2^{10} \\
%      - 555548672a_2^9 + 64536576a_2^8 + 395624448a_2^7 - 408748032a_2^6 + 171835392a_2^5 + 27197440a_2^4 - \\
%      57671680a_2^3 + 22806528a_2^2 - 3932160a_2 + 262144=0.\\
  \end{multline*}
  We also express $a_2$ in terms of powers of $a_1$ using LLL, but we do not write this expression as it requires $\sim 4000$ characters. Let $K/\Qq$ be the field extension defined by the minimal polynomial of $a_1$ and express $h$ as a linear form defined over $K$. Now, working over $K$ (defined symbolically as an abstract field) we will prove that $X\cap Z(h)$ is a pair of conics.
  
  We used {\tt Magma}~\cite{magma} to show that the singular subvariety $S$ of $X\cap Z(h)$ is of degree $6$ and that the reduced subvariety of $S$ is a degree $2$ irreducible point over $K$. Therefore, over $\Cc$, the curve $X\cap Z(h)$ must contain two singular points with the Jacobian of the curve cutting each of them out with multiplicity three. This means the singularities can not be nodes or cusps and, by the degree--genus formula for plane curves, the curve can not be irreducible. Since $X$ contains no lines, $X \cap Z(h)$ is a pair of conics bitangent to one another.
  
  Each of the $28$ embeddings of $K$ into $\Cc$ will give another pair of conics, proving that we have at least $56$ conics in $X$.

  We can now prove that $X$ has Picard number $8$. For $B >0$ let $\pic(X)|_B=\Zz \langle \xi \in \pic(X) \mid -\xi_{\prim}^2 < B \rangle$\emre{actually this is not the notation}. The lattice $\Lambda \subset \H^2(X,\Zz)$ contains $\pic(X)|_B$ for some explicit $B \gg 10^{1000}$~\cite{lairez-sertoz}. Therefore, we could not have missed the class of any conic in $X$. As we found $56$ conics in $X$ and $56$ conic classes in $\Lambda$, all of the conic classes in $\Lambda$ are truly in $\pic(X)$. As $\Lambda$ is integrally generated by these conic classes, we have $\Lambda=\pic(X)|_B$, which implies $\rk \pic(X) \ge \rk \Lambda = 8$. 
  
  By computing the Zeta function of the reduction of $X$ over $\Ff_{101}$ we show $\rk \pic(X) \le 8$~\cite{abbott-10,chk}. For this last step we used {\tt controlledreduction}\footnote{\url{https://github.com/edgarcosta/controlledreduction}}~\cite{costa-phd}.%checked
\end{proof}

%5*x^4 - 4*x^2*z*w + 8*y^4 - 5*z^4 + 4*z*w^3
\begin{proposition}
  The quartic surface $X=Z(5x^4 - 4x^2zw + 8y^4 - 5z^4 + 4zw^3 ) \subset \ppp_{\Cc}$ has Picard number~$14$. In fact, $X$ contains $102$ conics and $4$ lines whose classes generate $\pic(X)$.
\end{proposition}
\begin{proof}
  The general outline of the proof follows that of Proposition~\ref{prop:rk8}. We will mention the peculiarities of this quartic. First, it is readily checked that $X$ contains four coplanar lines. In particular, any plane containing a smooth conic of $X$ will contain another smooth conic.  The set of planes containing a conic in $X$ break into three Galois orbits. The first orbit consists of $24$ planes, each of the form $Z(x+a_2z+a_3w)$ and containing a pair of bitangent conics. The second orbit also has $24$ planes, but of the form $Z(y+a_2z+a_3w)$ and containing a pair of transversal conics. The third orbit has size $3$, elements of the form $Z(z+a_3w)$, $25a_3^3 + 4a_3 + 20$, containing a pair of bitangent conics. The Zeta function for the reduction was computed over the prime $101$.
\end{proof}

  \subsection{Twisted cubics in quartic surfaces revisited}\label{sec:twisted_cubics_revisited}

  \subsubsection{A single example}

  Consider the quartic surface cut out by $x^4 + x^3z + y^4 + z^4 + w^4$. This has Picard number 18, contains 16 lines, 288 conics and 1536 twisted cubics. We computed the $\tI_{\delta,2}$ for the class $\delta$ for each of the twisted cubics and found $\dim \tI_{\delta,2} = 3$ in each case. We thus have a floating point representation of the coordinates of $\tI_{\delta,2} \subset S_2$ in the Grassmannian $\gr(3,10)$. As these coordinates must be algebraic numbers, we may once again reconstruct them exactly recovering $\tI_{\delta,2}$ and therefore $I(T)$.

  \subsubsection{Many examples}\label{sec:many_examples}

  Now let $T \subset X$ be any twisted cubic in a smooth quartic. By general considerations, the ideal $\tI_{\delta}$ gives rise to an Artinian--Gorenstein quotient $S/\tI_{\delta}$ of socle degree~$4$~\cite[Remark~4]{loyola18}. Upto translation of $\ppp_\C$ we can put $T$ into the standard form and randomly sample quartics $X$ containing $T$, so that we know $I':=\jac(X)+I(T)$ but not $I_{[T]}$. It turns out that, in tens of thousands of random examples, $\codim I'_4=1$. If $g$ is the quartic form apolar to $I'$ then $I_{[T]}$ is the apolar ideal of $g$ (using Macaulay's correspondance theorem for graded Artinian--Gorenstein algebras~\cite[\S 1.3]{dolgachev}). In each of our random examples we thus computed $I_{[T],2}$ and observed that $\dim I_{[T],2} =3$. This point of view will be investigated in a future paper.

%  \subsubsection{Lines and genus one curves}
%
%  We may also compute lines in $x^4 + x^3z + y^4 + z^4 + w^4$ and find correctly that $\dim \tI_{\delta,1}=2$ for each class $\delta$ containing a line. Each line $L$ determines a pencil of genus one curves, all lying in one class $[E]$. Since primitive parts of the classes $[L]$ and $[E]$ span the same line, we have $\tI_{[L],1}=\tI_{[E],1}$. Indeed, the class $[E]$ of a plane cubic inside the quartic can not distinguish between different representatives so we find all planes that contain any $E' \in [E]$.

%  \begin{verbatim}
%  Higher dimensional examples coming soon!! 
%  It should be possible to either:
%  1) Randomly put together generators until some things have non-zero kernel
%  (an interesting way to discover meaningful classes)
%  2) Search for complete intersection classes (e.g. a plane in a fourfold),
%  its self-intersection can be computed so we can search for ``feasible'' classes then check ck1
%  \end{verbatim}

\section{Fields generated by periods}

Let $X$ be a smooth projective variety defined over a subfield $\k$ of $\Cc$ and let $\bark \subset \Cc$ be the algebraic closure of $\k$ in $\Cc$. Denote by $X^\an$ the underlying complex analytic manifold of $X$. In the following, $m\in\N$ is an even number. 

\begin{definition}
  A homology class $\delta \in \H_{m}(X^\an,\Q)$ and its Poincar\'e dual are called 
  \emph{algebraic} if there exists subvarieties $Z_1,\dots,Z_a \subset X$ of dimension $\frac{m}{2}$ 
  such that $\delta$ is a rational linear combination of $[Z_i]$.% In this case, $\delta$ and its Poincar\'e dual are said to be \emph{supported on $Z_1,\dots,Z_a$}. 
\end{definition}

The de~Rham cohomology can be defined purely algebraically~\cite{gro66} giving rise to the $\k$-vector spaces $\H_{\dR}^m(X/\k)$ equipped with a canonical isomorphism $\H_{\dR}^m(X/\k)\otimes_\k \C \simeq \H^m(X^\an, \C)$ for each $m \ge 0$. For a field extension $\sfK/\k$ we will write $\H_{\dR}^m(X/\sfK)$ for the algebraic de~Rham cohomology of the base change of $X$ to $\sfK$; recall that there is a natural isomorphism $\H_{\dR}^m(X/\sfK) \simeq \H_{\dR}^m(X/\k) \otimes_\k \sfK$. For the general discussion on algebraic de Rham cohomology see Deligne's lecture notes~\cite{dmos}. Our point of departure is the following fact.

\begin{proposition}[Proposition~1.5~\cite{dmos}]\label{prop:algebraic_periods}
  If $\delta \in \H_m(X,\Qq)$ is algebraic, then for every $\omega \in \H^m_{\dR}(X/\k)$ the following \emph{period integrals} are in~$\bark$:
\begin{equation}
  \label{18.10.2018}
  p_{\omega}(\delta):=\frac{1}{(2\pi i)^{\frac{m}{2}}}\mathlarger{\int}_\delta\omega.
\end{equation}
\end{proposition}

\begin{definition}
  Given an algebraic cycle $\delta=[Z]$ we will denote by $\k_\delta$, and by $\k_Z$, the field spanned over $\k$ by the set of all period integrals $p_\omega(\delta)$. 
\end{definition}

\subsection{Fields of definition of algebraic cycles}

The field $\k_\delta$ is the smallest field containing $\k$ such that $\delta$ induces a cycle class in $H^{2n-m}_\dR(X/\k_\delta)$. Therefore, if $\delta=[Z]$ then the field of definition of $Z$ must contain $\k_\delta$. It is possible that $Z$ is defined over a transcendental extension of $\k$, e.g., if $Z$ is a transcendental fiber of a continuous family of algebraic cycles inside $X$. 

\begin{proposition}
\label{prop:vistachinesa2018}
For any algebraic cycle $\delta \in \H_m(X^{\an},\Zz)$, there are subvarieties $Z_1,Z_2,\ldots, Z_s \subset X$ of pure dimension $\frac{m}{2}$ defined over $\k_\delta$ and integers $a_0,\dots,a_s,\ a_0>0$ such that 
\begin{equation}\label{eq:vista}
 a_0 \delta=a_1 [Z_1]+a_2 [Z_2]+\dots +a_s [Z_s]. 
\end{equation}
\end{proposition}
\begin{proof}
  We may write $\delta=[Z]$ for an algebraic cycle $Z$ on $X_{\Cc}$ and arguing as in the proof of~\cite[\nopp I.1.5]{dmos} we may assume $Z$ is defined over a finite extension $\k'$ of $\k$. As noted earlier, $\k_\delta \subset \k'$ is forced and we may assume $\k'/\k_\delta$ is Galois. Let $G$ be the Galois group $\Gal(\k'/\k_\delta)$.
We define the algebraic cycle
\begin{equation} \label{eq:vista_galois}
  Z' :=\frac{1}{\left|G \right|}\sum_{\sigma\in G} \sigma(Z)
\end{equation}
which is invariant under $G$ and therefore defined over $\k_\delta$. As the cohomology class $[Z]$ is defined over $k_{\delta}$ it is fixed by $G$. Therefore, $\forall \sigma \in G$, $[\sigma Z] = \sigma[Z] = [Z]$ and $[Z']=[Z]$ in $\H_m(X^{\an},\Qq)$.

To get the expression~\eqref{eq:vista} in $\H_m(X^{\an},\Zz)$ clear denominators in \eqref{eq:vista_galois} to get an identity modulo torsion. A sufficient multiple of this identity will kill the torsion, giving an equality.
\end{proof}

The statement of Proposition~\ref{prop:vistachinesa2018} is false if we require $a_0=1$, see Remark~\ref{rem:invariant_L}. What estimates can we put on the minimal possible $a_0$?  We provide an answer in the setting of the following theorem. 

\begin{theorem}
\label{thm:linear_system}
Let $X$ be a smooth projective variety over $\k\subset\C$ with $H^1(X,\O_X)=0$, $d_1$ an integer annihilating torsion in $\H^2(X^{\an},\Zz)$ and $d_2$ an integer such that $X$ has a rational point over a field extension $\k'/\k$ of degree $d_2$. For any divisor in $Z$ on $X_{\Cc}$ the following are true.
 \begin{enumerate}
   \item The divisor $d_1d_2 Z$ is linearly equivalent to a divisor defined over $\k_Z$. \label{item:integral_equivalence-2}
   \item If $d_1 Z$ is effective then  the family of divisors linearly equivalent to $d_1 Z$ are parametrized by a Brauer--Severi variety defined over $\k_Z$. \label{item:BS_variety}
 \end{enumerate}
\end{theorem}
\begin{proof}
Let $\cL=\co_{X^\an}(Z)$ be the line bundle corresponding to $Z$.
%and   $X/\bark$ be the base change of $X$ to $\bark$. 
Our hypothesis $H^1(X_,\O_X)=0$ implies that $\pic(X/\bark)=\pic(X^\an) \toi \H^2(X^{\an},\Zz)$ so that $\cL$ is defined over $\bark$.
By Poincar\'e duality we have 
\[
  p_\omega(Z)=\frac{1}{(2\pi i)^ {n-1}}\mathlarger{\int}_{[Z]}\omega= \frac{1}{(2\pi i)^n}\mathlarger{\int}_{X^{an}}\omega\cup c_1(\cL)=\Tr( \omega\cup 
  c_1(\cL)),
\]
where $c_1(\cL)$ is $2\pi i$ times the topological Chern class of $\cL$, and in fact $c_1(\cL)\in H^2_{\dR}(X/\bark)$.  
The cup product $H^{2n-2}(X/\k)\times H^{2}(X/\k)\to H^{2n}(X/\k)$ and the trace map 
$\Tr\colon H^{2n}(X/\k)\cong \k$ can be defined for $X/\k$. 
The composition of these two maps is a non-degenerate 
$\k$-bilinear map and, by our hypothesis on the periods of $Z$, $\Tr(\cdot\cup c_1(\cL))$ has values in $\k_Z$, thus $c_1(\cL)\in H^2_{\dR}(X/\k_Z)$.
Since $\H^1(X,\co_X)=0$, the kernel of the Galois invariant map $c_1\colon \pic(X/\bark) \to \H^2_{\dR}(X/\bark)$ is contained in the torsion subgroup of $\H^2(X^{\an},\Zz)$. Therefore, $\cL^{d_1} \in \pic(X/\bark)^{\Gal(\bark/\k_Z)}$.  This does not mean $\cL^{d_1}$ is defined over $\k_Z$, see Remark~\ref{rem:invariant_L}.  However, if $d_1 Z$ is effective then the $\bark$-projective space $|\H^0(X/\bark,\cL^{d_1})|$ descends to a Brauer--Severi variety defined over $\k_Z$, see~\cite{BruinFlynn2004,gille06}. Item~\ref{item:BS_variety} is now proven. 
  
As for Item~\ref{item:integral_equivalence-2}, the map $\pic(X/ \k') \to \pic(X/\bark)^{\Gal(\bark/\k')}$ is surjective~\cite[Corollary~2.3]{CorayManoil1996}. Moreover, via~\cite[Proposition~2.2]{CorayManoil1996} the cokernel of $\pic(X/ \k_Z) \to \pic(X/\bark)^{\Gal(\bark/\k_Z)}$ is annihilated by $[\k':\k_Z]$, which divides $d_2=[\k':\k]$. Therefore,  $(\cL^{d_1})^{d_2}$ is contained in $\pic(X/\k_Z)$. Any rational section of $\cL^{d_1d_2}$ will give a divisor defined over $\k_Z$ linearly equivalent to $d_1d_2 Z$.
 \end{proof}

\begin{remark}\rm\label{rem:rigid_linear_system}
  With $X$ as in Theorem~\ref{thm:linear_system} and $d_1=1$, if $Z \subset X^\an$ is an effective divisor isolated in its linear system then $Z$ is defined over $\k_Z$ even if $X$ has no $\k_Z$ points. In this case, we use the fact that a zero dimensional Brauer--Severi variety is a point.
\end{remark}

\begin{remark}\rm 
  A smooth surface $X$ in $\ppp_\k$ fulfills the hypotheses of Theorem~\ref{thm:linear_system} with $d_1=1$ and $d_2=\deg X$: the first follows by Lefschetz hyperplane theorem and the latter by intersecting $X$ with a line. 
\end{remark}

\begin{remark}\rm
\label{rem:invariant_L}
Galois invariance of the isomorphism class of a line bundle does not imply that a line bundle in the isomorphism class can be defined over the expected base field. Take the curve $X$ defined by $x^2+y^2+z^2$ in $\pp_\Q$ which has no $\Q$ points and therefore has no line bundles of degree $1$. On the other hand, since $\pic(X/\bark) \simeq \Zz$, the isomorphism class of a line bundle depends only on the degree, therefore the isomorphism class of line bundles of degree $1$ is fixed by the Galois action. The Brauer--Severi variety associated to this class is the curve $X/\Q$ itself.
\end{remark}

\subsection{Bounds on the degree of the field extension $\k_Z$}\label{sec:bounds}

For a smooth algebraic variety $X$ over $\k$ let $Z$ be an algebraic cycle of dimension $\frac{m}{2}$ in $X/\bark$. In this section we will give bounds on the degree $[\k_Z:\k]$.

Let $\omega_1,\omega_2,\ldots,\omega_b$ be a basis for the primitive cohomology $\PH^m_\dR(X/\k)$ and consider the period vector $p(Z):=\left ( p_{\omega_1}(Z),\ p_{\omega_2}(Z),\cdots, p_{\omega_b}(Z)\right) \in \bark^b$,
where we used the notation and statement of Proposition~\ref{prop:algebraic_periods}. Let $\Lambda_Z \subset \bark^b$ be the $\Zz$-module generated by the Galois conjugates of $p(Z)$ and let $\lambda_Z = \rk \Lambda_Z$. Our bounds rely on the following theorem. 

\begin{theorem}[Minkowski~\cite{minkowski87, guralnick06}]\label{thm:minkowski}
  The size of the largest finite group in $\gl(n,\Zz)$ divides 
\begin{equation}
  M(n):= \prod_{p \text{ prime}} p^{\sum_{k=0}^{\infty} \lfloor n/p^k(p-1) \rfloor}.
\end{equation}
\end{theorem}

\begin{lemma}\label{lem:bound}
  For any integer $\lambda \ge \lambda_Z$, the degree $[\k_Z : \k]$ is bounded above by $M(\lambda)$.
\end{lemma}
\begin{proof}
  Let $\sfK_Z/\k$ be the Galois closure of $\k_Z$ in $\Cc$ and let $G=\Gal(\sfK_Z/\k)$. Note that $[\k_Z : \k]$ divides $[\sfK_Z : \k]=|G|$ and any bound for the latter will bound the former.
  
 Let $v_1,v_2,\ldots, v_{\lambda_Z} \in \sfK_Z^b$ be a $\Zz$-basis for the saturation of $\Lambda_Z$. We thus get the following Galois representation:
 \begin{align*}%\label{eq:galois_representation}
   h\colon \Gal(\sfK_Z/\k)&\to \GL (\lambda_Z,\Z), \\
   h(\sigma)\cdot \left( v_1, v_2 , \dots , v_{\lambda_Z}  \right) &= \left(  \sigma(v_1) , \sigma(v_2) , \dots , \sigma(v_{\lambda_Z}) \right) . 
\end{align*}
It is easy to see that $\sfK_Z$ is generated over $\k$ by all the entries of $v_i$'s, therefore the representation $h$ is faithful. Consequently, $|G|=|h(G)|$ is bounded above by $M(\lambda_Z)$ as a result of Theorem~\ref{thm:minkowski}. The observation $M(\lambda_Z) \le M(\lambda)$ finishes the proof.
\end{proof}

\begin{lemma}\label{lem:bound2}
  Let $X\subset \Pp^{n+1}_{\k}$ be a smooth projective variety, $\rho=\rk \Hdg^m(X)$ and $\delta \in \Hdg^m(X)$ an effective Hodge cycle. Then the field extension $\k_\delta / \k$ has degree at most $M(\rho-1)$. 
\end{lemma}
\begin{proof}
  Let $h$ be the hyperplane class on $X$ and project $\Hdg^m(X)$ to $(h^m)^{\perp}$. The Galois orbit of the projection of $\delta$ has dimension at most $\rho-1$. The primitive periods of $\delta$ and of its projection are the same. Now apply Lemma~\ref{lem:bound}. 
\end{proof}

\begin{theorem}\label{thm:upper_bound}
  Let $(X,d_1,d_2)$ be as in Theorem~\ref{thm:linear_system} and let $\rho$ be the Picard number of $X$. Then $\pic(X)$ admits a $\Qq$-basis $[Z_i], i=1,2,..,\rho,$ where each $Z_i$ is defined over an extension of $\k$ of degree at most $M(\rho-1)$. If $d_1=d_2=1$ then a $\Zz$-basis of this form exists.
\end{theorem}
\begin{proof}
  Let $Y_1,\dots,Y_\rho \in \pic(X_{\Cc})$ be a $\Qq$-basis. Using Theorem~\ref{thm:linear_system} we can find divisors $Z_1,\dots,Z_\rho$ where $Z_i$ is defined over $\k_{Y_i}$ and $[Z_i] = d_1d_2[Y_i]$. By Lemma~\ref{lem:bound2} we conclude $[\k_{Y_i}:\k] \le M(\rho-1)$. If $d_1=d_2=1$ then begin with an integral basis $\{Y_i\}_{i=1}^{\rho} \subset \pic(X_\Cc)$.
\end{proof}

\begin{remark}
  The values of $M(n)$ for $n=1,\dots,20$ are listed below:
  \begin{multline*}
    2, 24, 48, 5760, 11520, 2903040, 5806080, 1393459200, 2786918400,\\
    367873228800, 735746457600, 24103053950976000, 48206107901952000,\\
    578473294823424000, 1156946589646848000, 9440684171518279680000, \\
    18881368343036559360000, 271211974879377138647040000, \\
    542423949758754277294080000, 3579998068407778230140928000000.
  \end{multline*}
\end{remark}

\begin{proposition}\label{prop:rk2_surface}
  If $X\subset\ppp_{\k}$ is a smooth surface of degree $d$ with Picard number $2$ then there exists a curve $C \subset X$ defined over a field extension $\k'/\k$ of degree at most $2d$ such that $[C]$ generates $\pic(X)$ together with the hyperplane class. 
\end{proposition}

\begin{proof}
  Let $v_1,v_2 \in \pic(X) \simeq \Zz^2$ be a basis and let $h \in \pic(X)$ be the hyperplane class. Writing $h=a_1v_1+a_2v_2$ we note $\gcd(a_1,a_2)=1$ since $h$ is indivisible. Take $b_1,b_2 \in \Zz$ such that $b_1a_1+b_2a_2=1$ and let $w=b_1v_1+b_2v_2$ so that $h,w$ is a basis for $\pic(X)$. By replacing $w$ with $w+ch$ for $c \gg 1$ we may assume $w$ is effective. Lemma~\ref{lem:bound2} implies that the field $\k_w$ is at most a quadratic extension of $\k$. If $X(\k) \neq \emptyset$ we can apply Theorem~\ref{thm:linear_system} for the existence of a curve $C/\k_w$ representing $w$.  Otherwise, intersecting $X$ with a line we see that $X$ admits a point over an extension of $\k$ of degree at most~$d$.
\end{proof}

\printbibliography
\end{document}